\documentclass{amsart}%
\usepackage{amsfonts}
\usepackage{amsmath}
\usepackage{amssymb}
\usepackage{color}
\usepackage[unicode=true]{hyperref}
\usepackage{graphicx}%
\providecommand{\U}[1]{\protect\rule{.1in}{.1in}}
\newtheorem{theorem}{Theorem}
\theoremstyle{plain}

\newtheorem{corollary}{Corollary}

\newtheorem{definition}{Definition}
\newtheorem{example}{Example}

\newtheorem{lemma}{Lemma}

\newtheorem{remark}{Remark}

\numberwithin{equation}{section}
\begin{document}
\title[Garsia-Rodemich spaces]{Garsia-Rodemich spaces: Local maximal functions and Interpolation}
\author{Sergey Astashkin}
\address{Department of Mathematics, Samara National Research University, Moskovskoye
shosse 34, Samara 443086, Russia}
\email{astash56@mail.ru}
\author{Mario Milman}
\address{Instituto Argentino de Matematica, Argentina}
\email{mario.milman@gmail.com}
\urladdr{https://sites.google.com/site/mariomilman}
\thanks{SA was supported by the Ministry of Education and Science of the Russian
Federation, project 1.470.2016/1.4, and by the RFBR grant 18-01-00414a.}
\thanks{MM was partially supported by a grant from the Simons Foundation (%
$\backslash$%
\#207929 to Mario Milman)}
\thanks{This paper is in final form and no version of it will be submitted for
publication elsewhere.}

\begin{abstract}
We characterize the Garsia-Rodemich spaces associated with a rearrangement
invariant space via local maximal operators. Let $Q_{0}$ be a cube in $R^{n}$.
We show that there exists $s_{0}\in(0,1),$ such that for all $0<s<s_{0},$ and
for all r.i. spaces $X(Q_{0}),$ we have%
\[
GaRo_{X}(Q_{0})=\{f\in L^{1}(Q_{0}):\Vert f\Vert_{GaRo_{X}}\simeq\Vert
M_{s,Q_{0}}^{\#}f\Vert_{X}<\infty\},
\]
where $M_{s,Q_{0}}^{\#}$ is the Str\"{o}mberg-Jawerth-Torchinsky local maximal
operator. Combined with a formula for the $K-$functional of the pair
$(L^{1},BMO)$ obtained by Jawerth-Torchinsky, our result shows that the
$GaRo_{X}$ spaces are interpolation spaces between $L^{1}$ and $BMO.$ Among
the applications, we prove, using real interpolation, the monotonicity under
rearrangements of Garsia-Rodemich type functionals. We also give an approach
to Sobolev-Morrey inequalities via Garsia-Rodemich norms, and prove necessary
and sufficient conditions for $GaRo_{X}(Q_{0})=X(Q_{0}).$ Using packings, we
obtain a new expression for the $K-$functional of the pair $(L^{1},BMO)$.

\end{abstract}
\maketitle

\section{Introduction}

The starting point of this research is the celebrated John-Nirenberg Lemma
which we now recall. Let $Q_{0}\subset\mathbb{R}^{n}$ be a fixed
cube\footnote{A \textquotedblleft cube" in this paper will always mean a cube
with sides parallel to the coordinate axes. We normalize $Q_{0}$ to have
measure $1.$}, $1<p<\infty,$ the John-Nirenberg spaces $JN_{p}:=JN_{p}(Q_{0})$
consist of all functions $f\in L^{1}(Q_{0})$ such that
(cf. \cite{jn}, \cite{torchinsky})%
\begin{equation}
\left\Vert f\right\Vert _{JN_{p}}=\sup_{\{Q_{i}\}_{i\in N}\in P}\left\{
{\displaystyle\sum\limits_{i}}
\left\vert Q_{i}\right\vert \left(  \frac{1}{\left\vert Q_{i}\right\vert }%
\int_{Q_{i}}\left\vert f-f_{Q_{i}}\right\vert dx\right)  ^{p}\right\}
^{1/p}<\infty, \label{lajnp}%
\end{equation}
where{ }$f_{Q}=\frac{1}{\left\vert Q\right\vert }\int_{Q}fdx$ and%
\[
{P:=P(Q_{0})=\{\{Q_{i}\}_{i\in N}:Q_{i}\;\;\mbox{are subcubes of}\;\;Q_{0}%
\;\;\mbox{with pairwise disjoint interiors}\}.}%
\]

To understand the correct definition for $p=\infty,$ we proceed as follows.
For $\pi={\{Q_{i}\}_{i\in N}}\in P,$ let $f_{\pi}(x)=%
{\displaystyle\sum\limits_{i}}
\left(  \frac{1}{\left\vert Q_{i}\right\vert }\int_{Q_{i}}\left\vert
f-f_{Q_{i}}\right\vert dx\right)  \chi_{Q_{i}}(x),$ then
\[
\left\Vert f\right\Vert _{JN_{p}}=\sup_{\pi\in P}\left\Vert f_{\pi}\right\Vert
_{L^{p}(Q_{0})}.
\]
This justifies the definition: $JN_{\infty}:=JN_{\infty}(Q_{0})$ consists of
all functions $f\in L^{1}(Q_{0})$ such that%
\[
\left\Vert f\right\Vert _{JN_{\infty}}=\sup_{\pi\in P}\left\Vert f_{\pi
}\right\Vert _{L^{\infty}(Q_{0})}.
\]
It follows readily that%
\[
JN_{\infty}(Q_{0})=BMO(Q_{0}).
\]
The John-Nirenberg Lemma \cite{jn} implies the following embeddings
\begin{equation}
JN_{p}{(Q_{0})}\subset\left\{
\begin{array}
[c]{cc}%
L(p,\infty){(Q_{0})} & 1<p<\infty\\
e^{L}{(Q_{0})} & p=\infty.
\end{array}
\right.  ,\label{new1}%
\end{equation}
where $L(p,\infty){(Q_{0})}$ is the "weak" $L_{p}$-space and $e^{L}{(Q_{0})}$
is the Orlicz space of exponentially integrable functions. This result, and
the spaces involved, has been the object of intensive study over the years
and, in particular, the space $BMO$ now plays a very important role in
harmonic analysis. We refer to \cite{jn}, \cite{ber}, \cite{daf},
\cite{torchinsky}, for background, different proofs and extensive bibliographies.

In their paper, Garsia-Rodemich \cite{garro} proposed a very original
approach\footnote{While \cite{garro} contains many interesting results, and
indeed, has been widely quoted in the literature, their proposed approach to
(\ref{new1}) has remained largely unnoticed until very recently (cf.
\cite{fenicae2}).} to (\ref{new1}). It is based on the following idea: To
effectively compare $JN_{p}$ with $L(p,\infty)$, $1<p<\infty,$ a new class of
spaces was introduced in \cite{garro}. We shall say that\footnote{To describe
the original results we shall use a temporary notation.} $f\in G_{p}%
:=G_{p}(Q_{0})$ if and only if $f\in L^{1}(Q_{0}),$ and $\exists C>0$ such
that for all $\{Q_{i}\}_{i\in N}$ $\in P$ we have
\begin{equation}%
{\displaystyle\sum\limits_{i}}
\frac{1}{\left\vert Q_{i}\right\vert }\int_{Q_{i}}\int_{Q_{i}}\left\vert
f(x)-f(y)\right\vert dxdy\leq C\left(
{\displaystyle\sum\limits_{i}}
\left\vert Q_{i}\right\vert \right)  ^{1/p^{\prime}}, \label{dada}%
\end{equation}
where $1/p^{\prime}=1-1/p$, and we let%
\[
\left\Vert f\right\Vert _{G_{p}}=\inf\{C:\text{ such that (\ref{dada})
holds}\}.
\]

The connection between the $JN_{p}$ and $G_{p}$ spaces can be seen from the
readily verified computation%
\begin{equation}
\int_{Q}\left\vert f(x)-f_{Q}\right\vert dx\leq\frac{1}{\left\vert
Q\right\vert }\int_{Q}\int_{Q}\left\vert f(x)-f(y)\right\vert dxdy\leq
2\int_{Q}\left\vert f(x)-f_{Q}\right\vert dx.\label{extradulfa0}%
\end{equation}
Indeed, combining (\ref{extradulfa0}) with H\"{o}lder's inequality, we find
that for each $\{Q_{i}\}_{i\in N}$ $\in P,$ we have
\begin{align*}%
{\displaystyle\sum\limits_{i}}
\frac{1}{\left\vert Q_{i}\right\vert }\int_{Q_{i}}\int_{Q_{i}}\left\vert
f(x)-f(y)\right\vert dxdy &  \leq2%
{\displaystyle\sum\limits_{i}}
\int_{Q}\left\vert f(x)-f_{Q}\right\vert dx\\
&  =2%
{\displaystyle\sum\limits_{i}}
\left\vert Q_{i}\right\vert ^{1/p^{\prime}}(\left\vert Q_{i}\right\vert
^{1/p}\frac{1}{\left\vert Q_{i}\right\vert }\int_{Q}\left\vert f(x)-f_{Q}%
\right\vert dx)\\
&  \leq2\left\{
{\displaystyle\sum\limits_{i}}
\left\vert Q_{i}\right\vert \right\}  ^{1/p^{\prime}}\left\{
{\displaystyle\sum\limits_{i}}
\left\vert Q_{i}\right\vert \left(  \frac{1}{\left\vert Q_{i}\right\vert }%
\int_{Q_{i}}\left\vert f-f_{Q_{i}}\right\vert dx\right)  ^{p}\right\}  ^{1/p}.
\end{align*}
Consequently,%
\[
\left\Vert f\right\Vert _{G_{p}}\leq2\left\Vert f\right\Vert _{JN_{p}}.
\]
The remarkable fact is that we actually have (cf. \cite{garro} for the one
dimensional case and \cite{fenicae2} in general) that as sets%
\begin{equation}
G_{p}=L(p,\infty),1<p<\infty.\label{newA}%
\end{equation}
It is easy to see that the definition of $G_{p}$ also makes sense for $p=1$
and $p=\infty.$ Consider the case $p=\infty,$ then we let $p^{\prime}=1,$ and
we see that the definition (\ref{dada}) makes sense in this case and we have%
\begin{equation}
G_{\infty}=BMO.\label{new2}%
\end{equation}
In fact, since for any cube $Q,$ we have $\{Q\}\in P,$ it follows from
(\ref{extradulfa0}) that%
\[
\int_{Q}\left\vert f(x)-f_{Q}\right\vert dx\leq\frac{1}{\left\vert
Q\right\vert }\int_{Q}\int_{Q}\left\vert f(x)-f(y)\right\vert dxdy\leq
\left\Vert f\right\Vert _{G_{\infty}}\left\vert Q\right\vert ,
\]
yielding%
\[
\left\Vert f\right\Vert _{BMO}\leq\left\Vert f\right\Vert _{G_{\infty}}.
\]
On the other hand, if $f\in BMO,$ then, once again using (\ref{extradulfa0}),
we see that for any $\pi\in P,$%
\begin{align*}
\sum\limits_{Q\in\pi}\frac{1}{\left\vert Q\right\vert }\int_{Q}\int%
_{Q}\left\vert f(x)-f(y)\right\vert dxdy &  \leq2\sum\limits_{Q\in\pi}%
\frac{\left\vert Q\right\vert }{\left\vert Q\right\vert }\int_{Q}\left\vert
f(x)-f_{Q}\right\vert dx\\
&  \leq2\left\Vert f\right\Vert _{BMO}\sum\limits_{Q\in\pi}\left\vert
Q\right\vert .
\end{align*}
Consequently,%
\[
\left\Vert f\right\Vert _{G_{\infty}}\leq2\left\Vert f\right\Vert _{BMO}.
\]
For $p=1,$ we let $p^{\prime}=\infty,$ and
\[
\left\Vert f\right\Vert _{G_{1}}:=\sup_{\pi\in P}\sum\limits_{Q\in\pi}\frac
{1}{\left\vert Q\right\vert }\int_{Q}\int_{Q}\left\vert f(x)-f(y)\right\vert
dxdy.
\]
Then\footnote{When comparing $G_{p}$ spaces with \ other function spaces we
must take into account that for any constant $c,$ $\left\Vert f-c\right\Vert
_{G_{p}}=\left\Vert f\right\Vert _{G_{p}}.$} (modulo constants)%
\[
G_{1}=L^{1}.
\]
Indeed, it is easy to see that
\[
\left\Vert f\right\Vert _{G_{1}}\leq2\left\Vert f\right\Vert _{L^{1}}.
\]

Conversely, let $\widetilde{L^{1}}$ be the space $L^{1}$ modulo constants,
then since $\{Q_{0}\}\in P,$
\begin{align*}
\left\Vert f\right\Vert _{\widetilde{L^{1}}} &  \leq\left\Vert f-f_{Q_{0}%
}\right\Vert _{L^{1}}\\
&  \leq\frac{1}{\left\vert Q_{0}\right\vert }\int_{Q_{0}}\int_{Q_{0}%
}\left\vert f(x)-f(y)\right\vert dxdy\text{ \ (by (\ref{extradulfa0}))}\\
&  \leq\left\Vert f\right\Vert _{G_{1}}.
\end{align*}
For $p=\infty,$ the method of proof of (\ref{newA}) that was given in
\cite{fenicae2} also yields
\begin{equation}
G_{\infty}\subset L(\infty,\infty),\label{newB}%
\end{equation}
where $L(\infty,\infty)$ is the Bennett-DeVore-Sharpley space
\[
L(\infty,\infty)=\{f\in L^{1}(Q_{0}):\left\Vert f\right\Vert _{L(\infty
,\infty)}=\sup_{t}(f^{\ast\ast}(t)-f^{\ast}(t)\}<\infty\}
\]
(here, $f^{\ast}$ is the decreasing rearrangement of $f$ and $f^{\ast\ast
}(t):=\frac{1}{t}\int_{0}^{t}f^{\ast}(s)\,ds$). Together, (\ref{new2}) and
(\ref{newB}) therefore provide the improvement of the John-Nirenberg
inequality obtained by Bennett-DeVore-Sharpley \cite{bds}\footnote{It is well
known and easy to see that if $f\in L(\infty,\infty)$ then $f\in e^{L}.$ This
can be seen from%
\begin{align*}
f^{\ast\ast}(t)-f^{\ast\ast}(\left\vert Q_{0}\right\vert ) &  =\int%
_{t}^{\left\vert Q_{0}\right\vert }(f^{\ast\ast}(s)-f^{\ast}(s))\frac{ds}{s}\\
&  \leq\left\Vert f\right\Vert _{L(\infty,\infty)}\log\frac{\left\vert
Q_{0}\right\vert }{t}.
\end{align*}
Moreover, as shown by Bennett-DeVore-Sharpley \cite{bds}, $L(\infty,\infty)$
is the rearrangement invariant hull of $BMO.$}, namely%

\[
BMO\subset L(\infty,\infty).
\]

In \cite{pafa} it was shown how the Garsia-Rodemich spaces fit in the theory
of Sobolev embeddings and in \cite{G-R spaces} the Garsia-Rodemich
characterization of the weak $L^{p}$ spaces was used to provide a streamlined
proof of the embedding theorem for the Bourgain-Brezis-Mironescu space
$\mathbf{B}$ (cf. \cite{bbm}),
\[
\mathbf{B\subset}L(n^{\prime},\infty).
\]

In short, the Garsia-Rodemich spaces provide a framework that can be used to
study a number of classical problems in analysis. It was then natural to
consider the problem of extending the Garsia-Rodemich construction. In
particular, in view of the characterization of $L(p,\infty)$ provided by
(\ref{newA}), we ask: what other rearrangement invariant spaces can be
characterized via a suitable extension of the Garsia-Rodemich conditions? In
this direction the following generalization of the condition (\ref{dada}) was
proposed in \cite{G-R spaces}.

Let $X:=X(Q_{0})$ be a rearrangement invariant space; for a given integrable
function $f$ we consider the class $\Gamma_{f}$ of integrable functions
$\gamma$ such that for all $\{Q_{i}\}_{i\in\mathbb{N}}\in P$ it holds%
\begin{equation}%
{\displaystyle\sum\limits_{i\in\mathbb{N}}}
\frac{1}{\left\vert Q_{i}\right\vert }\int_{Q_{i}}\int_{Q_{i}}\left\vert
f(x)-f(y)\right\vert dxdy\leq%
{\displaystyle\sum\limits_{i\in I}}
\int_{Q_{i}}\gamma(x)dx. \label{dulfa}%
\end{equation}
To describe the corresponding enlarged class of spaces associated with these
conditions, it will be convenient to replace our temporary notation for the
$G-$spaces as follows. We let%
\[
GaRo_{X}:=GaRo_{X}(Q_{0})=\{f:\left\Vert f\right\Vert _{GaRo_{X}}<\infty\},
\]
where\footnote{We let $\left\Vert f\right\Vert _{GaRo_{X}}=\infty,$ if
$\Gamma_{f}=\varnothing,$ moreover we shall use the convention, $\left\Vert
\gamma\right\Vert _{X}=\infty$ if $\gamma\notin X.$}{
\begin{equation}
\left\Vert f\right\Vert _{GaRo_{X}}=\inf\{\left\Vert \gamma\right\Vert
_{X}:\gamma\in\Gamma_{f}\}. \label{dulfanorma}%
\end{equation}
It is shown in \cite[Corollary 1 and Remark 2]{G-R spaces} that, in this new
notation, we have}%
\[
GaRo_{L(p,\infty)}=G_{p},1<p<\infty.
\]
Moreover, at the end points it is easy to verify that
\[
GaRo_{L^{\infty}}=G_{\infty}=BMO
\]
and%
\[
GaRo_{L^{1}}=G_{1}=L^{1}.
\]
Thus, (\ref{newA}) and (\ref{newB}) now read%
\begin{align*}
GaRo_{L(p,\infty)}  &  =L(p,\infty),1<p<\infty,\\
GaRo_{L^{\infty}}  &  \subset L(\infty,\infty).
\end{align*}
More generally, the following generalization
holds for any r.i. space $X$ (cf. \cite{G-R spaces}),%
\begin{equation}
GaRo_{X}=X,\text{ if }0<\alpha_{X}\leq\beta_{X}<1, \label{bruni}%
\end{equation}
where $\alpha_{X},\beta_{X}$ are the Boyd indices of $X$ (cf. Section
\ref{subsec:boyd}).

The characterization (\ref{bruni}) is very satisfactory since it captures all
the main results at the level of $L^{p}$ spaces, $1<p<\infty$. However, the
methods of \cite{G-R spaces} are not adequate to understand what happens when
the Boyd indices are zero or one. In fact, the analysis of the end point cases
of (\ref{bruni}) seems to require a new set of ideas. In this paper we obtain
a new characterization of the Garsia-Rodemich spaces via the
Str\"{o}mberg-Jawerth-Torchinsky local maximal operators (cf. \cite{str},
\cite{JT}). Let $s\in(0,1)$, then,%
\[
M_{s,Q_{0}}^{\#}f(x):=\sup_{Q_{0}\supset Q\backepsilon x}\inf_{c\in\mathbb{R}%
}\inf\{\alpha\geq0:|\{y\in Q:\,|f(y)-c|>\alpha\}|<s|Q|\},\;x\in Q_{0},
\]
where {the supremum} is taken over all the cubes $Q$ contained in $Q_{0}$ such
that $x\in Q.$ One of our main results in this paper (cf. Theorem
\ref{theostro1} below) states that there exists $s_{0}\in(0,1)$ such that, for
all $0<s<s_{0},$ and for all r.i. spaces $X,$
\begin{equation}
\Vert f\Vert_{GaRo_{X}}\simeq\Vert M_{s,Q_{0}}^{\#}f\Vert_{X}, \label{mor1}%
\end{equation}
{where the implied constants are independent of $f$\footnote{The expression
$F\preceq G$ means that $F\leq c\cdot G$ for some constant $c>0$ independent
of all or of a part of arguments $F$ and $G$. If $F\preceq G$ and $G\preceq F$
we write: $F\simeq G$.}.}

This result not only allows us to study the limiting cases of (\ref{bruni})
but at the same time provides a connection of the $GaRo_{X}$ spaces and
classical harmonic analysis. In particular, in Theorem \ref{Cor2} we show a
significant improvement over (\ref{bruni})
\begin{equation}
\alpha_{X}>0\Rightarrow GaRo_{X}=X. \label{bruni0}%
\end{equation}
In fact, for a large class of r.i. spaces of \textbf{fundamental type (}cf.
Section \ref{sec:alpha}, Definition \ref{def:fun}) (\ref{bruni0}) is best
possible. In Section \ref{sec:alpha} we prove Corollary \ref{coro:coro}:{ for
every r.i. space $X$ of fundamental type}
%
\begin{equation}
GaRo_{X}=X\Leftrightarrow GaRo_{X}\subset X\Leftrightarrow\alpha_{X}>0.
\label{bruni1}%
\end{equation}

Another consequence of (\ref{mor1}) is the fact that the $GaRo_{X}$ spaces are
real interpolation spaces between $L^{1}$ and $BMO.$ For example, this can be
seen as a consequence of (\ref{mor1}) and the formula of the $K-$functional
for the pair $(L^{1},BMO)$ obtained by Jawerth-Torchinsky \cite{JT},%
\begin{equation}
K(t,f;L^{1},BMO)\simeq\int_{0}^{t}(M_{s,Q_{0}}^{\#}f)^{\ast}(u)du,\;\;t>0.
\label{jt1}%
\end{equation}

The characterization (\ref{mor1}) connects $GaRo_{X}$ spaces with classical
harmonic analysis. Let
\[
f_{Q_{0}}^{\#}(x)=\sup_{Q_{0}\supset Q\backepsilon x}\frac{1}{\left\vert
Q\right\vert }\int_{Q}\left\vert f(x)-f_{Q}\right\vert dx,\;\;x\in
Q_{0},\label{feffe2}%
\]
and f{or a r.i. space $X$ we define}%
\begin{equation}
X^{\#}=\{f:f_{Q_{0}}^{\#}\in X\} \label{feffe}%
\end{equation}
with%
\begin{equation}
\left\Vert f\right\Vert _{X^{\#}}=\left\Vert f_{Q_{0}}^{\#}\right\Vert _{X}.
\label{feffe1}%
\end{equation}
We show that (cf. Theorem \ref{cor1} below),%
\begin{equation}
\left\Vert f\right\Vert _{X^{\#}}\simeq\left\Vert f\right\Vert _{GaRo_{X}%
}\text{ if and only if }\beta_{X}<1. \label{mor2}%
\end{equation}
Moreover, we consider generalized Fefferman-Stein inequalities of the
form\footnote{The classical inequalities of Fefferman-Stein \cite{fs}
correspond to $X=L^{p},1<p<\infty.$}%
\[
\inf_{c,\text{ constant}}\left\Vert f-c\right\Vert _{X}\leq C\left\Vert
f\right\Vert _{X^{\#}},\label{fefstein}%
\]
and prove that this inequality holds if $\alpha_{X}>0$ (cf. Theorem
\ref{theostro})\footnote{For a different approach to Fefferman-Stein
inequalities in the more general setting of Banach function spaces we refer to
Lerner \cite{lrnr}.}.

It is of interest to remark here that the conditions on the indices that
appear in the results described above are connected with considerations
arising from interpolation theory. For example, to compare the spaces
$GaRo_{X}$ and $X^{\#}$ one needs to understand the relationship between the
sharp maximal operator $f_{Q_{0}}^{\#}$and the local maximal operator
$M_{s,Q_{0}}^{\#}f,$ and one way to achieve this is via the formula for the
$K-$functional for the pair $L^{1}$ and $BMO$ provided by (\ref{jt1}), and the
formula obtained by Bennett-Sharpley {(cf. Example \ref{exam:K} below)%
\begin{equation}
\left(  f_{Q_{0}}^{\#}\right)  ^{\ast}(t)\simeq\frac{1}{t}\int_{0}%
^{t}(M_{s,Q_{0}}^{\#}f)^{\ast}(u)du=:(M_{s,Q_{0}}^{\#}f)^{\ast\ast}(t).
\label{jt3}%
\end{equation}
From (\ref{jt1}) and (\ref{jt3}) we see that the relationship between the
sharp maximal operator and the local maximal operator is analogous to the
classical relationship between }$f^{\ast}$ and $f^{\ast\ast}$.
Moreover, if we write $GaRo_{L(p,\infty)}=L(p,\infty)=(L^{1},BMO)_{1/p^{\prime
},\infty}$ we see that as $p\rightarrow\infty$ we approach the space $BMO$,
thus we expect to lose \textquotedblleft rearrangement
invariance\textquotedblright, and this may help to explain the requirement
$\alpha_{X}>0,$ to be able to attain results of the form $GaRo_{X}=X.$

The connection between Garsia-Rodemich spaces and interpolation goes deeper.
In fact, the ideas associated with the construction of Garsia-Rodemich spaces
lead us to find a new formula for the $K-$functional associated with the pair
$(L^{1},BMO),$ using packings (cf. Section \ref{sec:K} below), which we
believe should be of interest when comparing pointwise averages, as one often
does in the theory of weighted norm inequalities. As a concrete application of
this circle of ideas {we show how one can use }interpolation methods to prove
the monotonicity under rearrangements of certain Garsia-Rodemich type
functionals (our approach should be compared with the one provided in
\cite{garro}).

Finally, returning to some of the original results of Garsia and his
collaborators, we show a simple proof of a Sobolev-Morrey embedding in Section
\ref{sec:Sob}.

We refer the reader to Section \ref{sec:back} and to the monographs \cite{bs},
\cite{LT}, \cite{KPS}, \cite{BK} and \cite{torchinsky} for background
information and notation.

Acknowledgement. We are very grateful to the referee for detailed and
constructive criticism that helped us improve the presentation of the paper.

\section{Background Information\label{sec:back}}

\subsection{Rearrangements}

Let $\left(  \Omega,\mu\right)  $ be a resonant Borel probability space (cf.
\cite[pag. 64]{bs}). For a measurable function $f:{\Omega}\rightarrow
\mathbb{R},$ the \textbf{distribution function} of $f$ is given by
\[
\lambda_{f}(t):=\mu\{x\in{\Omega}:\left\vert f(x)\right\vert >t\},\;\;t>0.
\]
The \textbf{decreasing rearrangement} $f^{\ast}$ of a measurable function $f$
is the right-continuous non-increasing function, {mapping $(0,1]$ into
}$\mathbb{[}0,\infty),${ which is equimeasurable with $f,$ i.e., satisfying
\[
\lambda_{f}(t)=\left\vert \left\{  s\in\lbrack0,1]:f^{\ast}(s)>t\right\}
\right\vert ,\;\;t>0,
\]
where $\left\vert \cdot\right\vert $ denotes the Lebesgue measure on $[0,1].$
It can be defined by the formula%
\[
f^{\ast}(s):=\inf\{t>0:\lambda_{f}(t)\leq s\},\text{ \ }s\in(0,1].
\]
}

The maximal average $f^{\ast\ast}(t)$ is defined by
\[
f^{\ast\ast}(t):=\frac{1}{t}\int_{0}^{t}f^{\ast}(s)ds=\frac{1}{t}\sup\left\{
\int_{E}\left\vert f(s)\right\vert d\mu:\mu(E)=t\right\}  ,\;\;0<t\le1.
\]

\subsection{Rearrangement invariant spaces\label{secc:ri}}

We recall briefly the basic definitions and conventions we use from the theory
of rearrangement invariant (r.i.) spaces, and refer the reader to the books
\cite{bs}, \cite{KPS} and \cite{LT} for a complete treatment. {In the next
definition we follow \cite{LT}.}

Let $X:=X({\Omega})$ be a Banach function space on $({\Omega},\mu)$, which is
either separable or has the Fatou property (the latter means that if
$f_{n}\geq0,$ $f_{n}\uparrow f,$ and $f\in X$, then $\left\Vert f_{n}%
\right\Vert _{X}\uparrow\left\Vert f\right\Vert _{X}$). We shall say that $X$
is a \textbf{rearrangement invariant} (r.i.) space, if $g\in X$ implies that
all $\mu-$measurable functions $f$ with $f^{\ast}=g^{\ast}$ also belong to $X$
and, moreover, $\Vert f\Vert_{X}=\Vert g\Vert_{X}$. For any r.i. space $X$ we have%

\[
L^{\infty}\subset X\subset L^{1},
\]
with continuous embeddings. Many of the familiar spaces we use in analysis are
examples of r.i. spaces, e.g. the $L^{p}$-spaces, Orlicz spaces, Lorentz
spaces, Marcinkiewicz spaces, etc.

{Let $M$ be an increasing convex function on $[0,\infty)$ such that $M(0)=0$.
The Orlicz space $L_{M}$ consists of all measurable functions $x(t)$ on
$[0,1]$ such that the function $M\left(  {|x(t)|}/{\lambda}\right)  \in L^{1}$
for some $\lambda>0$. It is equipped with the Luxemburg norm
\begin{equation}
{\Vert x\Vert}_{L_{M}}:=\inf\left\{  \lambda>0:\int\limits_{0}^{1}M\left(
\frac{|x(t)|}{\lambda}\right)  \,dt\leqslant1\right\}  . \label{orlicz}%
\end{equation}
In particular, if $M(u)=u^{p}$, $1\leqslant p<\infty$, we obtain usual $L^{p}%
$--spaces. }

Let $\varphi$ be an increasing concave function on $[0,1]$, with
$\varphi(0)=0$. The Mar\-cin\-kie\-wicz space ${\mathcal{M}}(\varphi)$
consists of all measurable functions $x(t)$ such that
\[
{\Vert x\Vert}_{{\mathcal{M}}(\varphi)}:=\sup_{0<s\leqslant1}{\frac
{\varphi(s)}{s}\cdot\int\limits_{0}^{s}{x^{\ast}(t)dt}}<\infty.
\]
The space $L(p,\infty)$, $1<p<\infty$, corresponds to taking $\varphi
(s)=s^{1/p}$.

Let $X({\Omega})$ be a r.i. space, then there exists a \textbf{unique} r.i.
space (cf. \cite[pag. 64]{bs}) (the \textbf{representation space} of
$X({\Omega})),$ $\bar{X}=\bar{X}(0,1)$ on{ $\left(  \left(  0,1\right)
,\left\vert \cdot\right\vert \right)  $,} such that%
\[
\Vert f\Vert_{X({\Omega})}=\Vert f^{\ast}\Vert_{\bar{X}(0,1)}.
\]
In what follows if there is no possible confusion we shall not distinguish
between $X$ and $\bar{X}$.

The following majorization principle, usually associated to the names
Hardy-Littlewood-P\'{o}lya-Calder\'{o}n (cf. \cite{calderon},
\cite[Proposition~2.a.8]{LT}), holds for r.i. spaces: if
\begin{equation}
\int_{0}^{t}f^{\ast}(s)ds\leq\int_{0}^{t}g^{\ast}(s)ds,\text{ for all }t>0,
\label{hardy}%
\end{equation}
\ then, for any r.i. space $\bar{X},$%
\[
\left\Vert f^{\ast}\right\Vert _{\bar{X}}\leq\left\Vert g^{\ast}\right\Vert
_{\bar{X}},
\]
or equivalently,%
\[
\left\Vert f\right\Vert _{X}\leq\left\Vert g\right\Vert _{X}.
\]

The \textbf{fundamental function\ }of $X$ is defined by
\[
\phi_{X}(s)=\left\Vert \chi_{\left[  0,s\right]  }\right\Vert _{\bar{X}%
},\text{ \ }0\leq s\leq1.
\]
We can assume without loss of generality that $\phi_{X}$ is concave (cf.
\cite{bs}). {For example, for an Orlicz space $L_{N}$ (cf. (\ref{orlicz})
above) we have, $\phi_{L_{N}}(t)=1/N^{-1}(1/t)$ and for a Marcinkiewicz space
}${\mathcal{M}}(\varphi),$ the corresponding fundamental function is given by
{$\phi_{{\mathcal{M}}(\varphi)}(t)=\varphi(t)$.}

\subsection{Boyd indices and Hardy operators\label{subsec:boyd}}

Let $X=X({\Omega})$ be an arbitrary r.i. space. Then the compression/dilation
operator $\sigma_{s}$ on $\bar{X}$, defined by
\[
\sigma_{s}f(t)=\left\{
\begin{array}
[c]{ll}%
f^{\ast}(\frac{t}{s}), & 0<t<s,\\
0, & s\leq t.
\end{array}
\right.  \label{verarriba}%
\]
is bounded on $\bar{X},$ and moreover (cf. \cite[\S \,2.4]{KPS})%
\begin{equation}
\Vert\sigma_{s}\Vert_{\bar{X}\rightarrow\bar{X}}\leq\max\{1,s\},\text{ for all
}s>0. \label{verabajo}%
\end{equation}
The \textbf{Boyd indices (c}f. \cite{boyd}\textbf{) }are defined by
\[
\alpha_{X}:=\lim_{s\rightarrow0+}\dfrac{\ln\Vert\sigma_{s}\Vert_{\bar
{X}\rightarrow\bar{X}}}{\ln s}\text{ \ \ and \ }\beta_{X}:=\lim_{s\rightarrow
\infty}\dfrac{\ln\Vert\sigma_{s}\Vert_{\bar{X}\rightarrow\bar{X}}}{\ln s}.
\]
{For each r.i. space $X$ we have $0\leq\alpha_{X}\leq\beta_{X}\leq1$.} For
example, it follows readily that $\alpha_{L^{p}}=\beta_{L^{p}}=\frac{1}{p}$
for all $1\leq p\leq\infty$.

It is known that the Boyd indices control the boundedness of the \textbf{Hardy
operators}, which are defined by
\[
Pf(t):=\frac{1}{t}\int_{0}^{t}f(s)ds;\text{ \ \ \ }Qf(t):=\int_{t}%
^{1}f(s)\frac{ds}{s}.
\]
In fact, it is well known that {(cf. \cite{boyd}, \cite[Theorems~2.6.6 and
2.6.8]{KPS})}:
\begin{equation}%
\begin{array}
[c]{c}%
P\text{ is bounded on }\bar{X}\text{ }\Leftrightarrow\beta_{\bar{X}}<1,\\
Q\text{ is bounded on }\bar{X}\text{ }\Leftrightarrow\alpha_{\bar{X}}>0.
\end{array}
\label{alcance}%
\end{equation}

\subsection{K-functionals and real interpolation\label{subsec:interpolation}}

Let $(A_{0},A_{1})$ be a compatible pair of Banach spaces. For all $f\in
A_{0}+A_{1},t>0,$ we define the \textbf{Peetre $K-$functional} as follows%
\[
K(t,f;A_{0},A_{1}):=\inf\{\left\Vert f_{0}\right\Vert _{A_{0}}+t\left\Vert
f_{1}\right\Vert _{A_{1}}:f=f_{0}+f_{1},f_{i}\in A_{i},i=0,1\}.
\]
{Let $\theta\in(0,1),$ $1\leq q\leq\infty.$ The interpolation spaces
$(A_{0},A_{1})_{\theta,q}$ are defined by%
\[
(A_{0},A_{1})_{\theta,q}:=\{f:f\in A_{0}+A_{1}\text{ s.t. }\left\Vert
f\right\Vert _{(A_{0},A_{1})_{\theta,q}}<\infty\},
\]
where
\[
\left\Vert f\right\Vert _{(A_{0},A_{1})_{\theta,q}}:=\left\{
\begin{array}
[c]{cc}%
\left\{  \int_{0}^{\infty}\left(  s^{-\theta}K(s,f;A_{0},A_{1})\right)
^{q}\frac{ds}{s}\right\}  ^{1/q} & ,\text{ if }q<\infty\\
\sup_{s>0}\{s^{-\theta}K(s,f;A_{0},A_{1})\} & ,\text{ if }q=\infty.
\end{array}
\right.
\]
}

\begin{example}
(Peetre-Oklander formula (cf. \cite[(1.28) pag. 298]{bs}, \cite{Oklander}):
For the pair $(L^{1},L^{\infty})$ the $K-$functional is given by
\begin{equation}
K(t,f;L^{1},L^{\infty})=\int_{0}^{t}f^{\ast}(u)du,\;\;t>0. \label{okl}%
\end{equation}

\end{example}

Let $M_{Q_{0}}$ be the maximal operator of Hardy-Littlewood,%
\begin{equation}
M_{Q_{0}}f(x):=\sup_{Q_{0}\supset Q\backepsilon x}\frac{1}{\left\vert
Q\right\vert }\int_{Q}|f(y)|\,dy,\;\;x\in Q_{0}. \label{maximal}%
\end{equation}
The maximal operator $M_{Q_{0}}$ is connected with $K(t,\cdot;L^{1},L^{\infty
})$ via the Herz-Stein inequalities (cf. \cite[Theorem $3.8,$ pag. 122]{bs}):%
\begin{equation}
\left(  M_{Q_{0}}f\right)  ^{\ast}(t)\simeq f^{\ast\ast}(t):=\frac{1}{t}%
\int_{0}^{t}f^{\ast}(u)du,\;\;0<t\le1. \label{extrajt1}%
\end{equation}

\begin{example}
\label{exam:K}For the pair $(L^{1},BMO)$ (we consider classes of equivalence
modulo constants), we have the following formula due to Bennett-Sharpley (cf.
\cite[(8.11) pag. 393]{bs}):%
\begin{equation}
K(t,f;L^{1},BMO)\simeq t\left(  f_{Q_{0}}^{\#}\right)  ^{\ast}(t),\;\;0<t\leq
1. \label{extrabmo1}%
\end{equation}
{Comparing this with the Jawerth-Torchinsky formula \eqref{jt1} we see the
equivalence \eqref{jt3}.}
\end{example}

{In what follows any constant appearing in inequalities and depending only on
the dimension $n$ will be referred to as absolute.}

\section{A new description of the Garsia-Rodemich spaces\label{secc:feff}}

In this section we give a new characterization of the Garsia-Rodemich spaces
using local maximal operators. To motivate our result it will be useful to
reformulate somewhat the definition of the $\Gamma_{f}$ classes (cf.
(\ref{dulfa}) above).

It follows from inequalities \eqref{extradulfa0} that for an integrable
function $\gamma$ to belong to $\Gamma_{f}$ it is equivalent to verify the
following condition: there exists a constant $C>0$ such that, for all subcubes
$Q\subset Q_{0},$ we have%
\[
\frac{1}{\left\vert Q\right\vert }\int_{Q}\left\vert f(x)-f_{Q}\right\vert
dx\leq\frac{C}{\left\vert Q\right\vert }\int_{Q}\gamma(x)dx,
\]
whence
\[
f_{Q_{0}}^{\#}(x)\leq CM_{Q_{0}}\gamma(x),\;\;x\in Q_{0}.
\]
{The idea behind our main result can be now summarized as follows: for every
$f\in L^{1},$ the Str\"{o}mberg-Jawerth-Torchinsky maximal function
$M_{s,Q_{0}}^{\#}f$ is an \textquotedblleft optimal\textquotedblright\ choice
of $\gamma$ from $\Gamma_{f}.$ }

\begin{theorem}
\label{theostro1}There exists $s_{0}\in(0,1)$, {depending only on dimension
$n$,} such that, for all $s\in(0,s_{0}),$ and every r.i. space $X,$ we have%
\begin{equation}
GaRo_{X}=\{f\in L^{1}:\left\Vert M_{s,Q_{0}}^{\#}f\right\Vert _{X}<\infty\}.
\label{stro4}%
\end{equation}
Moreover, with constants of equivalence, {depending on $n\in\mathbb{N}$ and
$s\in(0,s_{0})$},%
\begin{equation}
\left\Vert f\right\Vert _{GaRo_{X}}\simeq\left\Vert M_{s,Q_{0}}^{\#}%
f\right\Vert _{X}. \label{stro5}%
\end{equation}

\end{theorem}

For the proof we shall need the following

\begin{lemma}
\label{L0} (i) For every cube $Q\subset Q_{0}$, all $0<s<1$, $c\in\mathbb{R}$,
and each $f\in L^{1}(Q_{0})$ we have
\[
|\{y\in Q:\,M_{0,s,Q}^{\#}f(y)>\lambda\}|\leq\frac{4^{n}}{s}|\{y\in
Q:\,|f(y)-c|>\lambda\}|,\;\;\text{for all }\lambda>0.
\]

(ii) There exists $0<s_{0}<1$ such that for all $s\in(0,s_{0})$and all $f\in
L^{1}(Q_{0})$
\begin{equation}
M_{Q_{0}}(M_{s,Q_{0}}^{\#}f)(x)\leq\frac{2\cdot8^{n}}{s}f_{Q_{0}}%
^{\#}(x),\;\;x\in Q_{0}, \label{stro7}%
\end{equation}
where $M_{Q_{0}}$ is the maximal function of Hardy-Littlewood (cf.
(\ref{maximal})).
\end{lemma}

\begin{proof}
(i) Let $Q\subset Q_{0}$ be an arbitrary cube. If $M_{s,Q}^{\#}f(y)>\lambda$
for $y\in Q$, then there is a cube $Q^{\prime}\subset Q$ such that $y\in
Q^{\prime}$ and for all $c\in\mathbb{R}$
\[
|\{z\in Q^{\prime}:\,|f(z)-c|>\lambda\}|>s|Q^{\prime}|.
\]
Therefore, we have
\[
M_{Q}(\chi_{\{|f-c|>\lambda\}})(y)\geq\frac{1}{|Q^{\prime}|}\int_{Q^{\prime}%
}\chi_{\{|f-c|>\lambda\}}(z)\,dz>s
\]
(here, $M_{Q}$ is the maximal operator of Hardy-Littlewood, corresponding to
the cube $Q$). Hence,
\[
|\{y\in Q:\,M_{s,Q}^{\#}f(y)>\lambda\}|\leq|\{y\in Q:\,M_{Q}(\chi
_{\{|f-c|>\lambda\}})(y)>s\}|.
\]
Combining this estimate with the fact that $M_{Q}$ is of weak type $(1,1)$
(cf. \cite[Theorem~3.3.3]{bs}), we see that
\[
|\{y\in Q:\,M_{s,Q}^{\#}f(y)>\lambda\}|\leq\frac{4^{n}}{s}\Vert\chi
_{\{|f-c|>\lambda\}}\Vert_{L^{1}(Q)}=\frac{4^{n}}{s}|\{y\in Q:\,|f-c|>\lambda
\}|.
\]

(ii) Let $x\in Q_{0}$ and $Q\subset Q_{0}$ be an arbitrary cube such that
$x\in Q$. Denote by $2Q$ the cube with the same center as the cube $Q$ and
with double side length. Clearly, there is a cube $\tilde{Q}$ such that
$Q_{0}\cap(2Q)\subset\tilde{Q}\subset Q_{0}$ and $|\tilde{Q}|\le|2Q|$. In
particular, if $2Q\subset Q_{0}$, we take $\tilde{Q}=2Q$. Note that $\tilde
{Q}\supset Q$.

Further, for all $y\in Q$ we have
\[
M_{s,Q_{0}}^{\#}f(y)\leq M_{s,\tilde{Q}}^{\#}f(y)+R_{s,\tilde{Q}}^{\#}f(y),
\]
where the operator $R_{s,\tilde{Q}}^{\#}$ is defined in just the same way as
$M_{s,\tilde{Q}}^{\#}$ except that the supremum is now taken over all cubes
having non-empty intersection with the set $Q_{0}\setminus\tilde{Q}$. From the
preceding inequality it follows that
\begin{equation}
\frac{1}{|Q|}\int_{Q}M_{s,Q_{0}}^{\#}f(y)\,dy\leq\frac{1}{|Q|}\int%
_{Q}M_{s,\tilde{Q}}^{\#}f(y)\,dy+\frac{1}{|Q|}\int_{Q}R_{s,\tilde{Q}}%
^{\#}f(y)\,dy. \label{extra10}%
\end{equation}

Applying part (i) of this lemma to the cube $\tilde{Q}$ and using the
properties of the latter cube, we estimate the first integral from the
right-hand side of \eqref{extra10} as follows:
\[
\frac{1}{|Q|}\int_{Q}M_{s,\tilde{Q}}^{\#}f(y)\,dy\leq2^{n}\frac{1}{|\tilde
{Q}|}\int_{\tilde{Q}}M_{s,\tilde{Q}}^{\#}f(y)\,dy\leq\frac{8^{n}}{s}\frac
{1}{|\tilde{Q}|}\int_{\tilde{Q}}|f(y)-c|\,dy
\]
for any $c\in\mathbb{R}$. On the other hand, since the cube $\tilde{Q}$ is
fixed, for each $\varepsilon>0$ we can choose a constant $c$ such that
\[
\frac{1}{|\tilde{Q}|}\int_{\tilde{Q}}|f(y)-c|\,dy\leq(1+\varepsilon
)\inf_{c^{\prime}\in\mathbb{R}}\frac{1}{|\tilde{Q}|}\int_{\tilde{Q}%
}|f(y)-c^{\prime}|\,dy.
\]
Combining these inequalities with the definition of $f_{Q_{0}}^{\#}(x)$, we
infer that
\begin{equation}
\frac{1}{|Q|}\int_{Q}M_{s,\tilde{Q}}^{\#}f(y)\,dy\leq(1+\varepsilon
)\frac{8^{n}}{s}f_{Q_{0}}^{\#}(x). \label{extra11}%
\end{equation}

To estimate the second integral from the right-hand side of \eqref{extra10},
{we will use the following observation. For each cube $Q^{\prime}$ such that
$Q^{\prime}\subset Q_{0}$ from $Q^{\prime}\cap(Q_{0}\setminus\tilde{Q}%
)\neq\varnothing$ it follows that $Q^{\prime}\cap(\mathbb{R}^{n}%
\setminus(2Q))\neq\varnothing$. Therefore, then there is a cube $Q^{\prime
\prime}\subset Q_{0}$ such that $Q^{\prime\prime}\supset Q$ and $|Q^{\prime
\prime}|\leq3^{n}|Q^{\prime}|$ and so from the definition of the operators
$M_{s,Q}^{\#}$ and $R_{s,\tilde{Q}}^{\#}$ we see that}
\[
\sup_{y\in Q}R_{s,\tilde{Q}}^{\#}f(y)\leq\inf_{y\in Q}M_{s^{\prime},Q}%
^{\#}f(y),
\]
where $s^{\prime}=s3^{-n}$. Now since $x\in Q$, we obtain,
\[
\frac{1}{|Q|}\int_{Q}R_{s,2Q}^{\#}f(y)\,dy\leq M_{s^{\prime},Q}^{\#}%
f(x)\leq\frac{3^{n}}{s}f_{Q_{0}}^{\#}(x),
\]
where the last inequality follows from Chebyshev's inequality. Combining our
findings with \eqref{extra10} and \eqref{extra11}, we obtain
\[
\frac{1}{|Q|}\int_{Q}M_{s,Q_{0}}^{\#}f(y)\,dy\leq2(1+\varepsilon)\frac{8^{n}%
}{s}f_{Q_{0}}^{\#}(x).
\]
Taking the supremum over all cubes $Q\subset Q_{0}$ such that $x\in Q$, and
letting $\varepsilon\rightarrow0$ we achieve the desired inequality \eqref{stro7}.
\end{proof}

\begin{proof}
[Proof of Theorem \ref{theostro1}]Suppose that $f\in L^{1}$ is such that
$\left\Vert M_{s,Q_{0}}^{\#}f\right\Vert _{X}<\infty$ {for some $s\in(0,1)$.}
Recall that by \cite[Lemma~2.4]{L-04}, there exists $s_{0}=s_{0}(n)>0$ such
that, for all $0<s<s_{0},$ and for every cube $Q\subset Q_{0}$, we have
\begin{equation}
\int_{Q}|f-f_{Q}|\,dx\leq8\int_{Q}M_{s,Q_{0}}^{\#}f\,dx. \label{equ103}%
\end{equation}
Consequently, by (\ref{extradulfa0}), $16M_{s,Q_{0}}^{\#}f\,\in\Gamma_{f}.$
Thus, for each $s\in(0,s_{0})$
\[
\left\Vert f\right\Vert _{GaRo_{X}}\leq16\left\Vert M_{s,Q_{0}}^{\#}%
f\right\Vert _{X}.
\]

Conversely, let $f\in GaRo_{X}.$ Given $\varepsilon>0$ we can select
$\gamma\in\Gamma_{f}$ $\cap X$ such that%
\begin{equation}
\left\Vert \gamma\right\Vert _{X}\leq\left\Vert f\right\Vert _{GaRo_{X}%
}+\varepsilon. \label{stro6}%
\end{equation}
{From the fact that $\gamma\in\Gamma_{f}$ it follows that (see the observation
in the beginning of this section)
\begin{equation}
f_{Q_{0}}^{\#}(x)\leq M_{Q_{0}}\gamma(x),\;\;x\in Q_{0}. \label{EQ20}%
\end{equation}
}

Consequently, by (\ref{stro7}), for all $0<s<1$
\begin{equation}
M_{Q_{0}}(M_{s,Q_{0}}^{\#}f)(x)\leq\frac{2\cdot8^{n}}{s}M_{Q_{0}}%
\gamma(x),\;\;x\in Q_{0}. \label{stro8}%
\end{equation}
Taking rearrangements in (\ref{stro8}), and using Herz's rearrangement
inequality for the Hardy-Littlewood maximal operator (cf. (\ref{extrajt1})),
for each $0<s<1$ we can find a constant $c=c(n,s)$ such that%
\[
\int_{0}^{t}(M_{s,Q_{0}}^{\#}f)^{\ast}(s)ds\leq c\int_{0}^{t}\gamma^{\ast
}(s)ds,\text{ for all }t>0.
\]
Hence, using successively the Hardy-Littlewood-P\'{o}lya-Calder\'{o}n
majorization principle (cf. (\ref{hardy})) and inequality (\ref{stro6}), we
get%
\begin{align*}
\left\Vert M_{s,Q_{0}}^{\#}f\right\Vert _{X}  &  \leq c\left\Vert
\gamma\right\Vert _{X}\\
&  \leq c\left\Vert f\right\Vert _{GaRo_{X}}+c\varepsilon.
\end{align*}
At this point we can let $\varepsilon\rightarrow0$ to obtain the desired
converse inequality.
\end{proof}

From Theorem \ref{theostro1}, and its proof, we readily obtain the following
alternative description of the Garsia-Rodemich spaces. Denote by $\Gamma
_{f}^{\prime}$ the set of all functions $\gamma\in L^{1}(Q_{0})$ satisfying \eqref{EQ20}.

\begin{corollary}
\label{description of GR} Let $X$ be a r.i. space. Then the Garsia-Rodemich
space $GaRo_{X}$ consists of all functions $f\in L^{1}(Q_{0})$ for which
$\Gamma_{f}^{\prime}\cap X\neq\varnothing$. Moreover, there exists an absolute
constant $c=c(n)$ such that,
\[
\inf\{\Vert\gamma\Vert_{X}:\,\gamma\in\Gamma_{f}^{\prime}\cap X\}\leq\Vert
f\Vert_{GaRo_{X}}\leq c\inf\{\Vert\gamma\Vert_{X}:\,\gamma\in\Gamma
_{f}^{\prime}\cap X\}.
\]

\end{corollary}

\section{A characterization of rearrangement invariant spaces via
Garsia-Rodemich conditions\label{sec:alpha}}

The main result of this section is the following characterization of r.i.
spaces which improves on (\ref{bruni}) above.

\begin{theorem}
\label{Cor2} Let $X$ be a r.i. space such that $\alpha_{X}>0.$ Then,%
\[
GaRo_{X}=X.
\]

\end{theorem}

\begin{proof}
Let $f\in X.$ Since for all cubes $Q\subset Q_{0}$ we have%
\[
\frac{1}{\left\vert Q\right\vert }\int_{Q}\int_{Q}\left\vert
f(x)-f(y)\right\vert dxdy\leq2\int_{Q}\left\vert f(x)-f_{Q}\right\vert
dx\leq4\int_{Q}\left\vert f(x)\right\vert dx,
\]
it follows from (\ref{dulfa}) that $4\left\vert f\right\vert \in\Gamma_{f}.$
Consequently,{ the embedding $X\subset GaRo_{X}$ holds for every r.i. space
$X$ and moreover}%
\[
\left\Vert f\right\Vert _{GaRo_{X}}\leq4\left\Vert f\right\Vert _{X}.
\]

We now show that if $\alpha_{X}>0,$ then $GaRo_{X}\subset X.$ Let $f\in
GaRo_{X},$ and let $\gamma$ be an arbitrary element of $\Gamma_{f}.$ Then, we
have{ \eqref{EQ20}, which combined with (\ref{extrajt1}) implies%
\[
(f_{Q_{0}}^{\#})^{\ast}(t)\leq(M_{Q_{0}}\gamma)^{\ast}(t)\preceq\gamma
^{\ast\ast}(t):=\frac{1}{t}\int_{0}^{t}\gamma^{\ast}(u)\,du.
\]
Thus, from \eqref{okl} and \eqref{extrabmo1}, we get}
\begin{equation}
K(t,f;L^{1},BMO)\preceq K(t,\gamma;L^{1},L^{\infty}), \label{equ5}%
\end{equation}
where the implied constants are independent of $f$ and $\gamma.$ Fix
$p>1/\alpha_{X}$. {It is well known that (cf. \cite[Theorem $8.11,$ pag
398]{bs})
\[
(L^{1},L^{\infty})_{\theta,p}=(L^{1},BMO)_{\theta,p}=L^{p},\;\;\theta
=1-\frac{1}{p}.
\]
Therefore, by Holmstedt's reiteration formula (cf. \cite[Corollary $2.3,$ pag
310]{bs}), we have
\[
K(t,f;L^{1},L^{p})\simeq t\Big(\int_{t^{1/\theta}}^{\infty}(s^{-\theta
}K(t,f;L^{1},BMO))^{p}\,\frac{ds}{s}\Big)^{1/p}%
\]
and
\[
K(t,\gamma;L^{1},L^{p})\simeq t\Big (\int_{t^{1/\theta}}^{\infty}(s^{-\theta
}K(t,\gamma;L^{1},L^{\infty}))^{p}\,\frac{ds}{s}\Big)^{1/p},
\]
with constants that depend only on $p$ (and hence on $X$). Combining these
estimates with \eqref{equ5} yields
\[
K(t,f;L^{1},L^{p})\preceq K(t,\gamma;L^{1},L^{p}),
\]
with constants that depend only on $X$ and $n$. Since the pair $(L^{1},L^{p})$
is $K$-monotone (cf. \cite{sparr}, \cite[Theorem~4]{cwikel})\footnote{A
different formulation of this result is given in {\cite[Theorem~3]{LS}.}}, it
follows that there exists a bounded linear operator $T$ acting on the pair
$(L^{1},L^{p}),$ such that $f=T\gamma$. Moreover, from the fact that
$p>1/\alpha_{X}$, we can deduce that $X$ is an interpolation space with
respect to the pair $(L^{1},L^{p})$ (cf. \cite[Theorem~2]{am-04}).
Consequently, by the $K$-monotonicity of $(L^{1},L^{p})$, there exists a
Banach lattice $(\Phi,\left\Vert .\right\Vert _{\Phi})$ of Lebesgue measurable
functions on $(0,\infty),$ such that the norm of $X$ can be represented as
follows (cf. \cite[Theorems 4.4.5 and 4.4.38]{BK})
\begin{equation}
\Vert x\Vert_{X}\simeq\Vert K(t,x;L^{1},L^{p})\Vert_{\Phi},\text{ for all
}x\in X. \label{nece}%
\end{equation}
It follows that the operator $T$ is bounded on $X$ and, consequently,
\[
\Vert f\Vert_{X}\leq c\Vert\gamma\Vert_{X},
\]
for{ some constant $c=c(n,X)$.} Taking the infimum over all $\gamma\in
\Gamma_{f}$, yields
\[
\Vert f\Vert_{X}\leq c\Vert f\Vert_{GaRo_{X}},
\]
as we wished to show.}
\end{proof}

Theorem \ref{Cor2} has a partial converse. To state the result we introduce
the class of r.i. spaces of fundamental type.

\begin{definition}
\label{def:fun}Let $X=X(Q_{0})$ be a r.i. space on $Q_{0}$, and let $\bar
{X}=X(0,1)$ be its Luxemburg representation on $(0,1)$ (cf. Section
\ref{secc:ri}). We shall say that $X$ is of fundamental type if there exists a
constant $C>0,$ such that (cf. Section \ref{subsec:boyd} above)
\[
\Vert\sigma_{t}\Vert_{\bar{X}\rightarrow\bar{X}}\leq C\sup_{s>0,st\leq1}%
\frac{\phi_{X}(st)}{\phi_{X}(s)},\;\;t>0.
\]

\end{definition}

\begin{remark}
It is easy to verify that Orlicz, Lorentz, Marcinkiewicz spaces, etc., are all
of fundamental type.
\end{remark}

\begin{definition}
\label{def::fun} A \textbf{median value}\footnote{Note that $m_{f}(Q)$ is not
uniquely defined.} of $f$ on $Q$ is a number $m_{f}(Q)$ such that
\[
|\{x\in Q:\,f(x)>m_{f}(Q)\}|\le\frac12|Q|
\]
and
\[
|\{x\in Q:\,f(x)<m_{f}(Q)\}|\le\frac12|Q|.
\]

\end{definition}

It is well known that $m_{f}(Q)$ is one of the constants $c$ minimizing some
functionals depending on the deviation $|f-c|$. In particular, we have (cf.
\cite[\S \,2, p.~2450]{L-04})
\[
\left(  f-m_{f}(Q_{0})\right)  ^{\ast}(t)\leq2\inf_{c\in\mathbb{R}}\left(
f-c\right)  ^{\ast}(t),\;\;0<t\leq1/2.
\]
From this inequality one can easily deduce that for every r.i. space $X$ the
following inequality holds:
\begin{equation}
\Vert f-m_{f}(Q_{0})\Vert_{X}\leq4\inf_{c\in\mathbb{R}}\Vert f-c\Vert_{X}
\label{optimality mediana}%
\end{equation}

\begin{theorem}
\label{T5} Let $X$ be a r.i. space of fundamental type, and let $0<s\leq1/2$.
{If there exists a constant $C>0$ such that%
\begin{equation}
\inf_{c\in\mathbb{R}}\Vert f-c\Vert_{X}\leq C\Vert M_{s,Q_{0}}^{\#}f\Vert_{X}
\label{mediana}%
\end{equation}
holds for all $f\in L^{1}(Q_{0})${$,$} then we must have $\alpha_{X}>0$.}


\end{theorem}

\begin{proof}
To the contrary, suppose that $\alpha_{X}=0$. Since $X$ is of fundamental type
we can find two numerical sequences $\{u_{k}\}_{k\in\mathbb{N}},\{a_{k}%
\}_{k\in\mathbb{N}}$ contained in $(0,1)$, converging to zero, and such that%
\begin{equation}
\phi_{X}(u_{k}a_{k})\geq\frac{1}{2}\phi_{X}(a_{k}),\;\;k=1,2,\dots\label{EQ12}%
\end{equation}

{Without loss of generality we can assume that $Q_{0}=[0,1]^{n}$. Moreover, if
$b>0$ we set $bQ_{0}:=[0,b]^{n}$. For $a\in(0,1),$ let $f_{a}(x):=n\ln
(\frac{a^{1/n}}{|x|_{\infty}})\chi_{a^{1/n}Q_{0}}(x)$, $x\ne0$, denoting
$|x|_{\infty}:=\max_{i=1,2,\dots,n}|x_{i}|$ for every $x=(x_{i})_{i=1}^{n}%
\in\mathbb{R}^{n}$. One can readily verify that there exists a constant
$D\geq1,$ that depends only on the dimension and $s$, such that $M_{s,Q_{0}%
}^{\#}f_{a}(x)\leq D$ if $|x|\leq Da$ and $M_{s,Q_{0}}^{\#}f_{a}(x)=0$ if
$|x|>Da$. Thus, using the concavity of the fundamental function $\phi_{X}$
(see Section \ref{secc:ri}), we get
\begin{equation}
\Vert M_{s,Q_{0}}^{\#}f_{a}\Vert_{X}\leq D\phi_{X}(Da)\leq D^{2}\phi
_{X}(a),\;\;0<a\leq1. \label{la M}%
\end{equation}
Moreover, it can be easily checked that $f_{a}^{\ast}(t)=\ln(a/t)\chi
_{(0,a)}(t)$ and $m_{f_{a}}(Q_{0})=0$ if $a$ is sufficiently small. Thus,
using (\ref{la M}), (\ref{mediana}), \eqref{optimality mediana} and
\eqref{EQ12}, for sufficiently large $k\in\mathbb{N}$, we have
\begin{align*}
D^{2}C\phi_{X}(a_{k})  &  \geq C\Vert M_{s,Q_{0}}^{\#}f_{a_{k}}\Vert_{X}\\
&  \geq\inf_{c\in\mathbb{R}}\Vert f_{a_{k}}-c\Vert_{X}\\
&  \geq\frac{1}{4}\Vert f_{a_{k}}-m_{f_{a_{k}}}(Q_{0})\Vert_{X}\\
&  =\frac{1}{4}\Vert f_{a_{k}}\Vert_{X}\\
&  \geq\frac{1}{4}\Vert\ln(a_{k}/t)\chi_{(0,a_{k})}(t)\Vert_{\bar{X}}\\
&  \geq\frac{1}{4}\Vert\ln(a_{k}/t)\chi_{(0,a_{k}u_{k})}(t)\Vert_{\bar{X}}\\
&  \geq\frac{1}{4}\ln(u_{k}^{-1})\Vert\chi_{(0,a_{k}u_{k})}(t)\Vert_{\bar{X}%
}\\
&  =\frac{1}{4}\ln(u_{k}^{-1})\phi_{X}(u_{k}a_{k})\\
&  \geq\frac{1}{8}\ln(u_{k}^{-1})\phi_{X}(a_{k}).
\end{align*}
This leads to a contradiction since $\lim_{k\rightarrow\infty}(\ln(u_{k}%
^{-1}))=\infty$. }
\end{proof}

Applying Theorems \ref{theostro1} --- \ref{T5}, we immediately obtain the
following result.

\begin{corollary}
\label{coro:coro}Let $X$ be an r.i. space of fundamental type. Then the
following conditions are equivalent:

(a) $GaRo_{X}=X$;

(b) $GaRo_{X}\subset X$;

(c) $\alpha_{X}>0$.
\end{corollary}

\section{$K-$functionals and rearrangement inequalities\label{sec:rea}}

In this section we consider some examples of the interaction of the
Garsia-Rodemich functionals with rearrangements, that are connected with our
development in this paper.

Our first application deals with a new proof of an inequality due to
Bennett-Sharpley (cf. \cite[Theorem $7.3,$ pag. 377]{bs}).

\begin{example}
\label{anterior}There exists an absolute constant $c>0,$ such that for all
$f\in L^{1}(Q_{0}),$ we have%
\begin{equation}
f^{\ast\ast}(t)-f^{\ast}(t)\leq c\left(  f_{Q_{0}}^{\#}\right)  ^{\ast
}(t),\;\;0<t<1/6. \label{rear0}%
\end{equation}

\end{example}

\begin{proof}
We recall the following fact from \cite{G-R spaces}: There exists an absolute
constant $c_{1}$ such that for all $f\in L^{1}(Q_{0}),$ and all $\gamma
\in\Gamma_{f}$, we have%
\begin{equation}
f^{\ast\ast}(t)-f^{\ast}(t)\leq c_{1}\gamma^{\ast\ast}(t),0<t<1/6.
\label{rear}%
\end{equation}
On the other hand, from (\ref{equ103}), we know that for sufficiently small
$s>0${ we have $16M_{s,Q_{0}}^{\#}f\in\Gamma_{f}.$} Consequently, by
(\ref{rear}),
\[
f^{\ast\ast}(t)-f^{\ast}(t)\leq16c_{1}\left(  M_{s,Q_{0}}^{\#}f\right)
^{\ast\ast}(t),\;\;0<t<1/6.
\]
Combining the last inequality with the fact that there exists an absolute
constant $c_{2}$ such that (cf. (\ref{jt3}))%
\[
\left(  M_{s,Q_{0}}^{\#}f\right)  ^{\ast\ast}(t)\leq c_{2}\left(  f_{Q_{0}%
}^{\#}\right)  ^{\ast}(t),
\]
we obtain (\ref{rear0}).
\end{proof}

Our second result shows how the continuity of rearrangements on
Garsia-Rodemich spaces can be easily established using their description
obtained in Theorem \ref{theostro1} and interpolation (compare with the
methods to establish related rearrangement inequalities that were developed in
\cite{garro} and \cite{beckner})).

\begin{theorem}
There exists an absolute constant $c>0$ such that for all $f\in GaRo_{X},$%
\[
\left\Vert f^{\ast}\right\Vert _{GaRo_{\bar{X}}(0,1)}\leq c\left\Vert
f\right\Vert _{GaRo_{X}(Q_{0})}.
\]

\end{theorem}

\begin{proof}
From \cite{garro}, \cite{bds} (cf. also \cite{cwikel2}), we know that there
exists an absolute constant $c_{1}\geq1,$ such that%
\[
\left\Vert f^{\ast}\right\Vert _{BMO(0,1)}\leq c_{1}\left\Vert f\right\Vert
_{BMO}.
\]
On the other hand, it is well known that (cf. \cite{garro},{
\cite[Theorem~2.3.1]{KPS}}) for all $f,g\in L^{1}(Q_{0}),$%
\[
\left\Vert f^{\ast}-g^{\ast}\right\Vert _{L^{1}(0,1)}\leq\left\Vert
f-g\right\Vert _{L^{1}(Q_{0})}.
\]
Consequently,{ for every $f\in L^{1}(Q_{0}),$}%
\begin{align*}
K(t,f^{\ast};L^{1}(0,1),BMO(0,1))  &  =\inf\{\left\Vert f_{1}\right\Vert
_{L^{1}(0,1)}+t\left\Vert f_{2}\right\Vert _{BMO(0,1)}:f^{\ast}=f_{1}%
+f_{2}\}\\
&  \leq\inf\{\left\Vert f^{\ast}-g^{\ast}\right\Vert _{L^{1}(0,1)}+t\left\Vert
g^{\ast}\right\Vert _{BMO(0,1)}:g\in BMO(Q_{0})\}\\
&  \leq\inf\{\left\Vert f-g\right\Vert _{L^{1}(Q_{0})}+tc_{1}\left\Vert
g\right\Vert _{BMO(Q_{0})}:g\in BMO(Q_{0})\}\\
&  \leq K(c_{1}t,f;L^{1}(Q_{0}),BMO(Q_{0}))\\
&  \leq c_{1}K(t,f;L^{1}(Q_{0}),BMO(Q_{0}))\text{ (since }K(t)/t\text{
decreases).}%
\end{align*}

In particular, in view of (\ref{jt1}), there exists an absolute constant
$c_{2}>0,$ such that%
\begin{equation}
\left(  M_{s,(0,1)}^{\#}f^{\ast}\right)  ^{\ast\ast}(t)\leq c_{2}\left(
M_{s,Q_{0}}^{\#}f\right)  ^{\ast\ast}(t). \label{ineq locmaxoper}%
\end{equation}
By the Hardy-Littlewood-P\'{o}lya-Calder\'{o}n principle, it follows that%
\begin{align*}
\left\Vert M_{s,(0,1)}^{\#}f^{\ast}\right\Vert _{\bar{X}(0,1)}  &  \leq
c^{\prime}\left\Vert (M_{s,Q_{0}}^{\#}f)^{\ast}\right\Vert _{\bar{X}(0,1)}\\
&  =c^{\prime}\left\Vert M_{s,Q_{0}}^{\#}f\right\Vert _{X(Q_{0})}.
\end{align*}
Applying (\ref{stro5}) we finally obtain
\[
\left\Vert f^{\ast}\right\Vert _{GaRo_{\bar{X}(0,1)}}\leq c\left\Vert
f\right\Vert _{GaRo_{X(Q_{0})}},
\]
as we wished to show.
\end{proof}

\begin{remark}
Essentially the same argument shows that if $T$ is a bounded operator on the
pair $(L^{1},BMO)$, then $T$ is a bounded operator in the space $GaRo_{X}.$
\end{remark}

\begin{proof}
Indeed, for such operators we have%
\[
K(t,Tf;L^{1},BMO)\leq cK(t,f;L^{1},BMO),\;\;t>0,
\]
which, in view of (\ref{jt1}), implies%
\[
\int_{0}^{t}\left(  M_{s,Q_{0}}^{\#}Tf\right)  ^{\ast}(s)ds\leq c\int_{0}%
^{t}\left(  M_{s,Q_{0}}^{\#}f\right)  ^{\ast}(s)ds.
\]
Therefore, we get (\ref{hardy}) and, as above, for any r.i. space $X$ we have%
\[
\left\Vert M_{s,Q_{0}}^{\#}Tf\right\Vert _{X}\leq c\left\Vert M_{s,Q_{0}}%
^{\#}f\right\Vert _{X}.
\]
The desired result now follows from Theorem \ref{theostro1}.
\end{proof}

\begin{remark}
As we have seen before (cf. (\ref{jt3})),$\left(  M_{s,(0,1)}^{\#}f^{\ast
}\right)  ^{\ast\ast}(t)\simeq\left(  \left(  f^{\ast}\right)  _{(0,1)}%
^{\#}\right)  ^{\ast}(t),$ and $\left(  M_{s,Q_{0}}^{\#}f\right)  ^{\ast\ast
}(t)\simeq\left(  f_{Q_{0}}^{\#}\right)  ^{\ast}(t),$ thus for a suitable
constant $C>0,${ from \eqref{ineq locmaxoper} it follows that}%
\[
\left(  \left(  f^{\ast}\right)  _{(0,1)}^{\#}\right)  ^{\ast}(t)\leq C\left(
f_{Q_{0}}^{\#}\right)  ^{\ast}(t),
\]
which should be compared with Theorem \ref{anterior}.
\end{remark}

\begin{remark}
The $K-$functional for the pair $(L^{\infty},BMO)$ was computed by several
authors including Janson, Jawerth-Torchinsky, Shvartsman (cf. \cite{JT},
\cite{shv} and the references therein). It would be of interest to connect the
interpolation spaces{ with respect to the pair }$(L^{\infty},BMO)$ and the
Garsia-Rodemich constructions.
\end{remark}

\section{Fefferman-Stein inequality via Garsia-Rodemich
spaces\label{secc:feff}}

The original Fefferman-Stein inequality (cf. \cite{fs} and also \cite{str} and
the references therein) concerns with the embedding (cf. (\ref{feffe}) and
(\ref{feffe1}) above)
\[
L^{p\#}\subset L^{p},1<p<\infty.
\]
{In \cite{str}, Str\"{o}mberg extended this result to an appropriate class of
Orlicz spaces.}

The connection between $X^{\#}$ and $GaRo_{X}$ can be seen from the fact that%
\begin{equation}
X^{\#}\subset GaRo_{X}. \label{stro2}%
\end{equation}
Indeed, we can easily show that {from $f\in X^{\#}$ it follows $2f_{Q_{0}%
}^{\#}\in\Gamma_{f}$.} {This follows directly from (\ref{extradulfa0}) since
for each $Q\subset Q_{0}$ we have%
\begin{align*}
\frac{1}{|Q|}\int_{Q}\int_{Q}|f(x)-f(y)|dxdy  &  \leq2\int_{Q}|f(x)-f_{Q}|dx\\
&  =2\frac{|Q|}{|Q|}\int_{Q}|f(x)-f_{Q}|dx\\
&  =2\int_{Q}\left(  \frac{1}{|Q|}\int_{Q}|f(x)-f_{Q}|dx\right)  dy\\
&  \leq2\int_{Q}f_{Q_{0}}^{\#}(y)dy,
\end{align*}
and so $\gamma:=2f_{Q_{0}}^{\#}$ satisfies inequality \eqref{dulfa}.
Consequently, (\ref{stro2}) holds for all r.i. spaces $X$, and, moreover, we
have%
\[
\left\Vert f\right\Vert _{GaRo_{X}}\leq2\left\Vert f\right\Vert _{X^{\#}}.
\]
}

Using the above observation, one can extend the Fefferman-Stein-Str\"{o}mberg
result\footnote{However, note that unlike \cite{str} we consider functions
defined on a fixed cube $Q_{0}.$} to the setting of r.i. spaces.

\begin{theorem}
\label{theostro} If the lower Boyd index $\alpha_{X}$ of the r.i. space $X$ is
positive, then $X^{\#}\subset X$.
\end{theorem}

\begin{proof}
From the condition $\alpha_{X}>0$ and Theorem \ref{Cor2} we infer that
$GaRo_{X}=X$. We conclude by combining this fact with \eqref{stro2}.
\end{proof}

The next result establishes necessary and sufficient conditions, under which
the opposite embedding $X\subset X^{\#}$ holds.

\begin{theorem}
\label{cor1}Let $X$ be an r.i. space on $[0,1]$. The following conditions are equivalent:

(i) $\beta_{X}<1$;

(ii) $GaRo_{X}\subset X^{\#}$;

(iii) $X\subset X^{\#}$.
\end{theorem}

\begin{proof}
$(i)\rightarrow(ii)$. Let $f\in GaRo_{X}.$ {As we have seen above for every
$\gamma\in\Gamma_{f}$, we have $f_{Q_{0}}^{\#}(x)\leq M_{Q_{0}}\gamma(x)$.}
Since we are assuming that $\beta_{X}<1$, the Hardy-Littlewood operator
$M_{Q_{0}}$ is bounded on $X$. Hence,
\[
\left\Vert f_{Q_{0}}^{\#}\right\Vert _{X}\leq\left\Vert M_{Q_{0}}%
\gamma\right\Vert _{X}\leq\Vert M_{Q_{0}}\Vert_{X\rightarrow X}\left\Vert
\gamma\right\Vert _{X}.
\]
Taking infimum over all $\gamma\in\Gamma_{f}$, we get
\[
\left\Vert f_{Q_{0}}^{\#}\right\Vert _{X}\leq\Vert M_{Q_{0}}\Vert
_{X\rightarrow X}\left\Vert f\right\Vert _{GaRo_{X}},
\]
whence $f\in X^{\#}.$

$(ii)\rightarrow(iii)$ The implication is trivial since the embedding
$X\subset GaRo_{X}$ holds for all r.i. spaces $X$ (see the beginning of the
proof of Theorem \ref{Cor2}).

$(iii)\rightarrow(i)$. By \cite[Theorem~5.7.3]{bs} (cf. also Example
\ref{anterior} in Section \ref{sec:rea}), we have
\[
f^{\ast\ast}(t)-f^{\ast}(t)\leq c^{\prime}(f_{Q_{0}}^{\#})^{\ast
}(t),\;\;0<t<1/6,
\]
{for some absolute constant $c^{\prime}$. Therefore,
\[
f^{\ast\ast}(t/6)\leq f^{\ast}(t/6)+c^{\prime}(f_{Q_{0}}^{\#})^{\ast
}(t/6),\;\;0<t<1.
\]
From the latter inequality, (\ref{verabajo}), and our current assumption, it
follows that
\begin{align*}
\Vert f^{\ast\ast}\Vert_{X}  &  \leq\Vert\sigma_{6}f^{\ast\ast}\Vert_{X}\\
&  \leq\Vert\sigma_{6}f\Vert_{X}+c^{\prime}\Vert\sigma_{6}f_{Q_{0}}^{\#}%
\Vert_{X}\\
&  \leq6c^{\prime}(\Vert f\Vert_{X}+\Vert f_{Q_{0}}^{\#}\Vert_{X})\\
&  \leq c\Vert f\Vert_{X}.
\end{align*}
This shows that the Hardy operator $P$ is bounded on $X$, and therefore, by
(\ref{alcance}), $\beta_{X}<1$.}
\end{proof}

\section{A packing formula for the $K-$functional of $(L^{1},BMO)$%
\label{sec:K}}

The new characterization of the Garsia-Rodemich spaces discussed in the
introduction (cf. (\ref{mor1}) above) suggested a new formula for the
$K-$functional of the pair $(L^{1},BMO)$ {(see Section
\ref{subsec:interpolation}).}

\begin{remark}
\label{rem: mean zero} In order to properly interpret the pair $(L^{1},BMO)$
as a compatible pair of Banach spaces, it is necessary to factor out the
constant functions. Equivalently, we can restrict ourselves to consider
functions with zero mean, i.e. $\int_{Q_{0}}f(x)dx=0.$
\end{remark}

For any family of cubes $\pi=\{Q_{i}\}\in P:=P(Q_{0}),$ we define
\[
S_{\pi,\sharp}(f)(x)=\sum_{Q_{i}\in\pi}\left(  \frac{1}{\left\vert
Q_{i}\right\vert }\int_{Q_{i}}|f(y)-f_{Q_{i}}|dy\right)  \chi_{Q_{i}%
}(x),\;\;x\in Q_{0},
\]
and let
\[
F_{f,\sharp}(t)=sup_{\pi\in P}(S_{\pi,\sharp}(f))^{\ast}(t),\;\;0<t\leq1.
\]

\begin{theorem}
\label{theokfunct}There exist absolute constants, such that for all $f\in
L^{1}$ we have
\[
K(t,f;L^{1},BMO)\simeq tF_{f,\sharp}(t),\;\;0<t\leq1.
\]

\end{theorem}

\begin{proof}
{It is plain that%
\[
F_{f,\sharp}(t)\leq f^{\sharp\ast}(t),\;\;0<t\leq1.
\]
Consequently, by equivalence \eqref{extrabmo1} (the implied constants depend
only on the dimension), we have%
\[
tF_{f,\sharp}(t)\preceq K(t,f;L^{1},BMO),\;\;0<t\leq1.
\]
}

Thus, the desired result will follow if we show that
\begin{equation}
K(t,f;L^{1},BMO)\preceq tF_{f,\sharp}(t),0<t\leq1 \label{equ4e}%
\end{equation}
with some absolute constant.

Given $t\in(0,1],$ we consider the set
\[
\Omega(t):=\{x\in Q_{0}:\,f^{\sharp}(x)>f^{\sharp\ast}(t)\}.
\]
It follows that for each $x\in\Omega(t)$ there exists a cube $Q_{x}$ such that
$Q_{x}\subset Q_{0}$, $x\in Q_{x}$, and
\begin{equation}
\frac{1}{|Q_{x}|}\int_{Q_{x}}|f-f_{Q_{x}}|>f^{\sharp\ast}(t). \label{equ2}%
\end{equation}
Note that, by the definition of the set $\Omega(t)$, we have $Q_{x}%
\subset\Omega(t)$ for every $x\in\Omega(t)$. Consider the family of cubes
$\{Q_{x}\}_{x\in\Omega(t)}.$ Using a Vitaly type covering lemma (cf.
\cite[p.~9]{Stein}), we can select a subfamily of pairwise disjoint cubes
$\{Q_{k}\}$ (which may contain a finite number of elements) such that
\begin{equation}
|\Omega(t)|=\Big|\bigcup_{x\in\Omega(t)}Q_{x}\Big|\leq5^{n}\sum_{k}|Q_{k}|.
\label{equ2extra}%
\end{equation}

Clearly $\pi=\{Q_{k}\}\in P$ and, moreover,{ by \eqref{equ2},}
\[
S_{\pi,\sharp}(f)(x)>f^{\sharp\ast}(t)\;\;\mbox{for all}\;\;x\in%
{\displaystyle\bigcup\limits_{k}}
Q_{k}.
\]
Therefore, combining \eqref{equ2extra} and the fact that $|\Omega(t)|\geq t$,
we obtain
\[
|\{x\in Q_{0}:\,S_{\pi,\sharp}(f)(x)>f^{\sharp\ast}(t)\}|\geq5^{-n}%
|\Omega(t)|\geq5^{-n}t.
\]
Thus, by the definition of the decreasing rearrangement of a measurable
function, it follows that,
\[
F_{f,\sharp}(5^{-n}t)\geq S_{\pi,\sharp}(f)^{\ast}(5^{-n}t)\geq f^{\sharp\ast
}(t),\;\;0<t\leq1.
\]
Equivalently,
\[
f^{\sharp\ast}(5^{n}t)\leq F_{f,\sharp}(t),\;\;0<t\leq5^{-n}.
\]
From the latter inequality, \eqref{extrabmo1} and the fact that
$K(t):=K(t,f;L^{1},BMO)$ is an increasing function, we have
\[
K(t)\leq K(5^{n}t)\simeq5^{n}tf^{\sharp\ast}(5^{n}t)\leq5^{n}tF_{f,\sharp
}(t),\;\;0<t\leq5^{-n}.
\]
Suppose now that $5^{-n}<t\leq1$. Let us first remark that $K(1)\leq\Vert
f\Vert_{L^{1}}.$ Indeed, we may assume that $\int_{Q_{0}}f(x)\,dx=0$ (see
Remark \ref{rem: mean zero}) and therefore to compute $K(1)$ we can use the
decomposition $f=f+0,$ and the assertion follows since%
\[
\Vert f\Vert_{L^{1}}=\frac{1}{\left\vert Q_{0}\right\vert }\int_{Q_{0}%
}\left\vert f-f_{Q_{0}}\right\vert dx\leq\Vert f\Vert_{BMO}.
\]
Let us also note that, since $\pi=\{Q_{0}\}\in P,$ we have $F_{f,\sharp
}(1)\geq\Vert f\Vert_{L^{1}}$. Consequently, using successively that $K(t)$ is
increasing, $F_{f,\sharp}(t)$ is decreasing, and $5^{n}t>1,$ we get
\[
K(t)\leq K(1)\leq\Vert f\Vert_{L^{1}}\leq F_{f,\sharp}(1)\leq5^{n}%
tF_{f,\sharp}(t).
\]
Thus, inequality \eqref{equ4e} holds for all $0<t\leq1$ with {constant
$c=5^{n}$.}
\end{proof}

\begin{remark}
Let $p\in(0,1)$. For any family of cubes $\pi=\{Q_{i}\}\in P(Q_{0})$ we let
\[
S_{\pi,\sharp}^{p}(f)(x):=\sum_{i}\left(  \frac{1}{|Q_{i}|}\int_{Q_{i}%
}|f-f_{Q_{i}}|^{p}\right)  ^{1/p}\chi_{Q_{i}}(x),
\]%
\[
F_{f,\sharp}^{p}(t):=sup_{\pi\in P}(S_{\pi,\sharp}^{p}(f))^{\ast}(t).
\]
Then, by a slight modification of the proof of Theorem \ref{theokfunct} we see
that the following equivalence holds
\[
K(t,f;L_{p},BMO)\simeq tF_{f,\sharp}^{p}(t),\;\;0<t\le1
\]
(cf. \cite[Remark~6.3]{BSh-79}).
\end{remark}

\section{Extensions of the Garsia-Rodemich construction\label{sec:Sob}}

We very briefly illustrate some of the results discussed in this paper showing
how adding a parameter to the Garsia-Rodemich construction leads to a
connection with the theory of Campanato spaces and the Morrey-Sobolev theorem.
We refer to \cite{adams} for more information and background.

\begin{definition}
Let $\lambda\in(-n,0],1<p\leq\infty.$ We shall say that $f\in L^{1}$ belongs
to $GaRo_{p,\lambda}$ if there exists a constant $C>0$ such that for all
$\{Q_{i}\}\in P,$
\end{definition}

\begin{equation}%
{\displaystyle\sum\limits_{i}}
\frac{1}{\left\vert Q_{i}\right\vert }\int_{Q_{i}}\int_{Q_{i}}\left\vert
f(x)-f(y)\right\vert dxdy\leq C\left(
{\displaystyle\sum\limits_{i}}
\left\vert Q_{i}\right\vert ^{1+\frac{\lambda}{n}}\right)  ^{1/p^{\prime}%
},\text{ where }1/p^{\prime}=1-1/p. \label{camp1}%
\end{equation}
and let%
\[
\left\Vert f\right\Vert _{GaRo_{p,\lambda}}:=\inf\{C:\text{(\ref{camp1})
holds}\}.
\]

Recall the definition of the homogeneous Campanato space $\mathcal{\dot{L}%
}^{1,\lambda}$ (cf. \cite[Section 2.2, pag 8]{adams})$:$

\begin{definition}
$\mathcal{\dot{L}}^{1,\lambda}=\{f:\left\Vert f\right\Vert _{\mathcal{\dot{L}%
}^{1,\lambda}}:=\sup_{Q\subset Q_{0}}\left\vert Q\right\vert ^{-\frac{\lambda
}{n}}(\frac{1}{\left\vert Q\right\vert }\int_{Q}\left\vert f-f_{Q}\right\vert
)<\infty\}.$
\end{definition}

\begin{theorem}
\label{campanato}$GaRo_{\infty,\lambda}=\left\{
\begin{array}
[c]{c}%
=\mathcal{\dot{L}}^{1,\lambda}\text{ },\text{ if }\lambda\in(-n,0)\\
= BMO,\text{ if }\lambda=0
\end{array}
\right.  .$
\end{theorem}

\begin{proof}
{Clearly, it is sufficiently to consider the case when $\lambda\in(-n,0)$.}

We will use repeatedly the fact that (see \eqref{extradulfa0})%
\[
\frac{1}{\left\vert Q\right\vert }\int_{Q}\int_{Q}\left\vert
f(x)-f(y)\right\vert dxdy\simeq\int_{Q}\left\vert f(x)-f_{Q}\right\vert \,dx.
\]
Consequently, we can write,%
\[
\left\Vert f\right\Vert _{\mathcal{\dot{L}}^{1,\lambda}}\simeq\sup_{Q\subset
Q_{0}}\left\vert Q\right\vert ^{-\frac{\lambda}{n}-1}\frac{1}{\left\vert
Q\right\vert }\int_{Q}\int_{Q}\left\vert f(x)-f(y)\right\vert dxdy.
\]
Suppose that $f\in GaRo_{\infty,\lambda}.$ Then, since for each $Q\subset
Q_{0}$ we have $\{Q\}\in$ $P,$ we see that%
\[
\frac{1}{\left\vert Q\right\vert }\int_{Q}\int_{Q}\left\vert
f(x)-f(y)\right\vert dxdy\leq\left\vert Q\right\vert ^{\frac{\lambda}{n}%
+1}\left\Vert f\right\Vert _{GaRo_{\infty,\lambda}}.
\]
Hence,%
\[
\left\Vert f\right\Vert _{\mathcal{\dot{L}}^{1,\lambda}}\preceq\left\Vert
f\right\Vert _{GaRo_{\infty,\lambda}}.
\]

Conversely, suppose that $f\in\mathcal{\dot{L}}^{1,\lambda}$ and let
$\{Q_{i}\}$ be an arbitrary element of $P.$ We compute,%
\begin{align*}%
{\displaystyle\sum\limits_{i}}
\frac{1}{\left\vert Q_{i}\right\vert }\int_{Q_{i}}\int_{Q_{i}}\left\vert
f(x)-f(y)\right\vert dxdy  &  =%
{\displaystyle\sum\limits_{i}}
\left\vert Q_{i}\right\vert ^{\frac{\lambda}{n}+1}\left\vert Q_{i}\right\vert
^{-\frac{\lambda}{n}-1}\frac{1}{\left\vert Q_{i}\right\vert }\int_{Q_{i}}%
\int_{Q_{i}}\left\vert f(x)-f(y)\right\vert dxdy\\
&  \preceq\left\Vert f\right\Vert _{\mathcal{\dot{L}}^{1,\lambda}}%
{\displaystyle\sum\limits_{i}}
\left\vert Q_{i}\right\vert ^{\frac{\lambda}{n}+1}.
\end{align*}
Consequently,%
\[
\left\Vert f\right\Vert _{GaRo_{\infty,\lambda}}\preceq\left\Vert f\right\Vert
_{\mathcal{\dot{L}}^{1,\lambda}}.
\]

\end{proof}

The import of the Campanato spaces stems from a well known result by Campanato
and Meyers (cf. \cite[(2.3), pag. 9]{adams}) showing that for $\lambda
\in(-1,0)$%
\begin{equation}
\mathcal{\dot{L}}^{1,\lambda}(Q_{0})=Lip(-\lambda)(Q_{0}). \label{meyer}%
\end{equation}

Let $\alpha\in(0,1),$ $p\geq1.$ Define,
\[
W^{\alpha,p}:=W^{\alpha,p}(Q_{0})=\{f:\left\Vert f\right\Vert _{W^{\alpha,p}%
}=\left\{  \int_{Q_{0}}\int_{Q_{0}}\frac{\left\vert f(x)-f(y)\right\vert ^{p}%
}{\left\vert x-y\right\vert ^{n+\alpha p}}dxdy\right\}  ^{1/p}<\infty\}.
\]
Then, we have the classical

\begin{theorem}
Let $p>\frac{n}{\alpha}$. Then
\[
W^{\alpha,p}\subset GaRo_{\infty,\frac{n}{p}-\alpha}=\mathcal{\dot{L}%
}^{1,\frac{n}{p}-\alpha}=Lip(\alpha-\frac{n}{p}).
\]

\end{theorem}

\begin{proof}
Note that $-1<\frac{n}{p}-\alpha<0$. In view of \ Theorem \ref{campanato},
\eqref{camp1} and (\ref{meyer}) for any cube $Q\subset Q_{0}$ we need to
estimate from above the quantity
\[
I:=\left\vert Q\right\vert ^{\frac{\alpha}{n}-\frac{1}{p}-1}\frac
{1}{\left\vert Q\right\vert }\int_{Q}\int_{Q}\left\vert f(x)-f(y)\right\vert
dxdy.
\]
We proceed as follows,%
\begin{align*}
I  &  \preceq\left\vert Q\right\vert ^{\frac{\alpha}{n}-\frac{1}{p}%
-2}\left\vert Q\right\vert ^{\frac{n+\alpha p}{np}}\int_{Q}\int_{Q}%
\frac{\left\vert f(x)-f(y)\right\vert }{\left\vert x-y\right\vert
^{\frac{n+\alpha p}{p}}}dxdy\\
&  \leq\left\vert Q\right\vert ^{\frac{\alpha}{n}-\frac{1}{p}-2+\frac{1}%
{p}+\frac{\alpha}{n}}\left\vert Q\right\vert ^{2(1-\frac{1}{p})}\left\{
\int_{Q}\int_{Q}\frac{\left\vert f(x)-f(y)\right\vert ^{p}}{\left\vert
x-y\right\vert ^{n+\alpha p}}dxdy\right\}  ^{1/p}\text{ (by H\"{o}lder's
inequality)}\\
&  \leq\left\vert Q\right\vert ^{2(\frac{\alpha}{n}-\frac{1}{p})}\left\Vert
f\right\Vert _{W^{\alpha,p}}\\
&  \leq\left\vert Q_{0}\right\vert ^{2(\frac{\alpha}{n}-\frac{1}{p}%
)}\left\Vert f\right\Vert _{W^{\alpha,p}},\\
&  \text{as we wished to prove.}%
\end{align*}

\end{proof}


\begin{thebibliography}{99}                                                                                               %


\bibitem {adams}D. Adams, \textsl{Morrey spaces}, Birkhauser, 2015.

\bibitem {am-04}S.V.~Astashkin and L.~Maligranda, \textsl{Interpolation
between $L_{1}$ and $L_{p}$, $1<p<\infty$,} Proc. AMS \textbf{132} (2004), 2929-2938.

\bibitem {beckner}W. Beckner, \textsl{Sobolev inequalities, the Poisson
semigroup, and analysis on the sphere $S^{n}$}, Proc. Natl. Acad. Sci. USA
\textbf{89} (1992), 4816-4819.

\bibitem {bds}C. Bennett, R. DeVore and R. Sharpley, \textsl{Weak-}$L^{\infty
}$\textsl{ and }$BMO$, Ann. Math. \textbf{113} (1981), 601-611.

\bibitem {BSh-79}C. Bennett and R. Sharpley, \textsl{Weak-type inequalities
for $H^{p}$ and $BMO$}, in Proc. \textquotedblleft Harm. Anal. Eucl.
Spaces\textquotedblright. AMS, Williamstown. Mass. 1978, (1979), 201--229.

\bibitem {bs}C. Bennett and R. Sharpley, \textsl{Interpolation of operators},
Academic Press, 1988.

\bibitem {ber}L. Berkovits, J. Kinnunen and J. M. Martell, \textsl{Oscillation
estimates, self-improving results and good-}$\lambda$\textsl{ inequalities},
J. Funct. Anal. \textbf{270} (2016), 3559-3590.

\bibitem {bbm}J. Bourgain, H. Brezis and P. Mironescu,\textsl{ A new function
space and applications}, J. Eur. Math. Soc. \textbf{17} (2015), 2083--2101.

\bibitem {boyd}D. W. Boyd, \textsl{Indices of function spaces and their
relationship to interpolation}, Canad. Math. J. \textbf{21} (1969), 1245-1254.

\bibitem {BK}Yu. A. Brudnyi and N. Ya.~Krugljak, \textsl{ Interpolation
Functors and Interpolation Spaces,} North Holland, Amsterdam, 1991.

\bibitem {calderon}A. P. Calder\'{o}n, \textsl{Spaces between }$L^{1}$\textsl{
and }$L^{\infty}$\textsl{ and the theorem of Marcinkiewicz}, Studia Math.
\textbf{26} (1966), 273-299.

\bibitem {cwikel}M. Cwikel, \textsl{Monotonicity properties of interpolation
spaces}, Ark. Mat.\textbf{14} (1976), 213--236.

\bibitem {cwikel2}M. Cwikel, Y. Sagher and P. Shvartsman \textsl{A new look at
the John--Nirenberg and John--Str\"{o}mberg theorems for BMO}, J. Funct. Anal.
\textbf{263} (2012), 129-166.

\bibitem {daf}G. Dafni, T. Hyt\"{o}nen, R. Kortec, and H. Yue,\textsl{ The
space JNp: Nontriviality and duality, }J. Funct. Anal. \textbf{275 }(2018), 577--603.

\bibitem {fs}C. Fefferman and E. M.~Stein, \textsl{$H^{p}$ spaces of several
variables}, Acta Math.\textbf{129} (1972), 137--193.



\bibitem {garro}A. M. Garsia and E. Rodemich, \textsl{Monotonicity of certain
functional under rearrangements}, Ann. Inst. Fourier (Grenoble) \textbf{24}
(1974), 67-116.

\bibitem {JT}B. Jawerth and A. Torchinsky, \textsl{Local Sharp Maximal
Functions}, J. Approx. Theory. \textbf{43} (1985), 231-270.

\bibitem {jn}F. John and L. Nirenberg, \textsl{On functions of bounded mean
oscillation}, Comm. Pure Appl. Math. \textbf{14} (1961), 415-426.

\bibitem {KPS}S. G. Krein, Yu. I. Petunin and E. M. Semenov,
\textsl{Interpolation of Linear Operators}, Amer. Math. Soc., Providence, 1982.

\bibitem {L-04}A. Lerner, \textsl{Weighted rearrangement inequalities for
local sharp maximal functions}, Trans. Amer. Math. Soc. \textbf{357} (2004), 2445--2465.

\bibitem {lrnr}A. Lerner, \textsl{Some remarks on the Fefferman-Stein
inequality}, Journal d'Analyse Math. \textbf{112} (2010), 329--349.

\bibitem {LT}J. Lindenstrauss and L. Tzafriri, \textsl{Classical Banach
Spaces, II. Function Spaces}, Springer-Verlag, Berlin-New York, 1979.

\bibitem {LS}G. G.~Lorentz and N.~Shimogaki, \textsl{Interpolation theorems
for the pairs of spaces $(L^{p},L^{\infty})$ and $(L^{1},L^{q})$,} Trans.
Amer. Math. Soc. \textbf{159} (1971), 207--221.

\bibitem {fenicae2}M. Milman, \textsl{Marcinkiewicz spaces, Garsia-Rodemich
spaces and the scale of John-Nirenberg self improving inequalities}, Ann.
Acad.\ Sci. Fennic\ae \ Math\textbf{. 41} (2016), 491--501.

\bibitem {pafa}M. Milman, A note on self-improvement of Poincar\'{e}-Sobolev
inequalities via Garsia-Rodemich spaces, Pure Appl. Funct. Anal. \textbf{1}
(2016), 429-439.

\bibitem {G-R spaces}M. Milman, \textsl{Garsia-Rodemich Spaces:
Bourgain-Brezis-Mironescu space, embeddings and rearrangement invariant
spaces}, Journal D'Analyse Math., to appear (preprint at https://arxiv.org/abs/1608.07849).

\bibitem {Oklander}E. Oklander, \textsl{On interpolation of Banach spaces},
PhD thesis, Univ. of Chicago, 1963.

\bibitem {shv}P. Shvartsman, \textsl{The }$K-$\textsl{functional of the pair
}$(L_{\infty}(w),BMO)$, Israel Math. Proc. 13 (1999), 183-203.

\bibitem {sparr}G. Sparr,\textsl{ Interpolation of weighted }$L^{p}$\textsl{
spaces}, Studia Math. \textbf{62} (1978), 229--271.

\bibitem {Stein}E. M. Stein, \textsl{Singular integrals and differentiability
properties of functions}. Princeton University Press, Princeton, 1970.

\bibitem {str}J.-O. Str\"{o}mberg, \textsl{Bounded mean oscillation with
Orlicz norms and duality of Hardy spaces}, Indiana Univ. J. \textbf{28}
(1979), 511-544.

\bibitem {torchinsky}A. Torchinsky, \textsl{Real variable methods in harmonic
analysis}. Academic Press, 1986.
\end{thebibliography}
\end{document}